\documentclass{amsart}
\usepackage{setspace}
\usepackage{a4}
\usepackage{amsthm}
\usepackage{latexsym}
\usepackage{amsfonts}
\usepackage{graphicx}
\usepackage{textcomp}
\usepackage{cite}
\usepackage{enumerate}
\usepackage{amssymb}
\usepackage{hyperref}
\usepackage{amsmath}
\usepackage{tikz}
\usepackage[mathscr]{euscript}
\usepackage{mathtools}
\newtheorem{theorem}{Theorem}[section]

\newtheorem{corollary}[theorem] {Corollary}
\newtheorem{definition}[theorem]{Definition}
\newtheorem{example}[theorem]{Example}

\newtheorem{problem}[theorem]{Problem}

\title{This is the title}
\usepackage{fancyhdr}

\begin{document}
\hrule\hrule\hrule\hrule\hrule
\vspace{0.3cm}	
\begin{center}
{\bf{$p$-adic Equiangular Lines and  $p$-adic van Lint-Seidel Relative Bound}}\\
\vspace{0.3cm}
\hrule\hrule\hrule\hrule\hrule
\vspace{0.3cm}
\textbf{K. Mahesh Krishna}\\
School of Mathematics and Natural Sciences\\
Chanakya University Global Campus\\
Haraluru Village, Near Kempe Gowda International Airport (BIAL)\\
Devanahalli Taluk, 	Bengaluru  Rural District\\
Karnataka State 562 110 India\\
Email: kmaheshak@gmail.com\\

Date: \today
\end{center}

\hrule\hrule
\vspace{0.5cm}
\textbf{Abstract}: We introduce the notion of $p$-adic equiangular lines and derive the first fundamental relation between common angle,  dimension of the space and the number of  lines. More precisely, we show that if $\{\tau_j\}_{j=1}^n$  is a collection of  $p$-adic  $\gamma$-equiangular lines in $\mathbb{Q}^d_p$, then 
\begin{align*}
(1)	 \quad\quad \quad \quad |n|^2\leq |d|\max\{|n|, \gamma^2 \}.
\end{align*}	
We  call  Inequality (1) as the $p$-adic van Lint-Seidel relative bound. We believe that this complements fundamental van Lint-Seidel  \textit{[Indag. Math., 1966]} relative bound for equiangular lines in the $p$-adic case.



\hrule

\hrule
\section{Introduction}
Let $d \in \mathbb{N}$ and $\gamma \in  [0,1]$. Recall that a collection $\{\tau_j\}_{j=1}^n$  of unit vectors  in $\mathbb{R}^d$	is said to be \textbf{$\gamma$-equiangular  lines} \cite{HAANTJES, LEMMENSSEIDEL} if 
	\begin{align*}
			|\langle \tau_j, \tau_k\rangle|=\gamma, \quad \forall 1\leq j, k \leq n, j \neq k.
	\end{align*}	
A fundamental problem associated with equiangular lines is  the following.
\begin{problem}\label{1} 
	Given $d \in \mathbb{N}$ and $\gamma \in  [0,1]$, what is the upper bound on $n$ such that there exists a collection $\{\tau_j\}_{j=1}^n$  of  $\gamma$-equiangular lines in $\mathbb{R}^d$?
	\end{problem}
An answer to Problem \ref{1}, which is   fundamental driving force in the study of equiangular lines, is the following result of van Lint and Seidel \cite{VANLINTSEIDEL, LEMMENSSEIDEL}.
\begin{theorem} \cite{VANLINTSEIDEL, LEMMENSSEIDEL} (\textbf{van Lint-Seidel Relative Bound}) \label{VSTHEOREM}
Let $\{\tau_j\}_{j=1}^n$  	be  $\gamma$-equiangular  lines in $\mathbb{R}^d$. 
Then 
\begin{align*}
n(1-d\gamma^2) \leq d(1-\gamma^2).
\end{align*}
In particular, if 
\begin{align*}
	\gamma <\frac{1}{\sqrt{d}},
\end{align*}
then 
\begin{align*}
	n\leq \frac{d(1-\gamma^2)}{1-d\gamma^2}.
\end{align*}
\end{theorem}
 While deriving $p$-adic Welch bounds, the notion of $p$-adic equiangular lines is hinted in \cite{MAHESHKRISHNA}. In this paper, we make  it more rigorous and derive a fundamental relation which complements Theorem \ref{VSTHEOREM}.

\section{$p$-adic Equiangular Lines}
We begin by recalling the   notion of  $p$-adic Hilbert space.  We refer  \cite{KALISCH, DIAGANABOOK, DIAGANARAMAROSON, KHRENNIKOV, ALBEVERIO} for more  information on $p$-adic Hilbert spaces.
\begin{definition}\cite{DIAGANABOOK, DIAGANARAMAROSON} \label{PADICDEF}
	Let $\mathbb{K}$ be a non-Archimedean  valued field with valuation $|\cdot|$ and $\mathcal{X}$ be a non-Archimedean Banach space with norm $\|\cdot\|$ over $\mathbb{K}$. We say that $\mathcal{X}$ is a \textbf{$p$-adic Hilbert space} if there is a map (called as $p$-adic inner product) $\langle \cdot, \cdot \rangle: \mathcal{X} \times \mathcal{X} \to \mathbb{K}$ satisfying following.
	\begin{enumerate}[\upshape (i)]
		\item If $x \in \mathcal{X}$ is such that $\langle x,y \rangle =0$ for all $y \in \mathcal{X}$, then $x=0$.
		\item $\langle x, y \rangle =\langle y, x \rangle$ for all $x,y \in \mathcal{X}$.
		\item $\langle \alpha x, y+z \rangle =\alpha (\langle x,  y \rangle+\langle x,z\rangle)$ for all  $\alpha  \in \mathbb{K}$, for all $x,y,z \in \mathcal{X}$.
		\item $|\langle x, y \rangle |\leq \|x\|\|y\|$ for all $x,y \in \mathcal{X}$.
	\end{enumerate}
\end{definition}
The following is the  standard example of a $p$-adic Hilbert space which we consider in the paper.
\begin{example}\cite{KALISCH}
	Let $p$ be a prime. For $d \in \mathbb{N}$, let $\mathbb{Q}_p^d$ be the standard $p$-adic Hilbert space equipped with the inner product 
	\begin{align*}
		\langle (a_j)_{j=1}^d,(b_j)_{j=1}^d\rangle := \sum_{j=1}^da_jb_j,  \quad \forall (a_j)_{j=1}^d,(b_j)_{j=1}^d \in \mathbb{Q}_p^d
	\end{align*}
	and the norm 
	\begin{align*}
		\|(x_j)_{j=1}^d\|:= \max_{1\leq j \leq d}|x_j|, \quad \forall (x_j)_{j=1}^d\in 	\mathbb{Q}_p^d.
	\end{align*}
\end{example}
Through various trials, we believe that the following is a correct definition of equiangular lines in the $p$-adic setting.
\begin{definition}\label{EL}
		Let $p$ be a prime, $d \in \mathbb{N}$ and $\gamma \geq0$. A collection $\{\tau_j\}_{j=1}^n$  in $\mathbb{Q}_p^d$	is said to be \textbf{$p$-adic $\gamma$-equiangular  lines} if the following conditions hold.
	\begin{enumerate}[\upshape(i)]
		\item $\langle \tau_j, \tau_j\rangle =1, \quad \forall 1\leq j \leq n$.
		\item $	|\langle \tau_j, \tau_k\rangle| =\gamma, \forall 1\leq j, k \leq n, j \neq k$.
		\item The operator 
		\begin{align*}
			S_\tau : \mathbb{Q}_p^d\ni x \mapsto \sum_{j=1}^n\langle x, \tau_j\rangle \tau_j \in \mathbb{Q}_p^d
		\end{align*}
	is similar  to a diagonal operator over $\mathbb{Q}_p$  with eigenvalues $\lambda_1, \dots, \lambda_d \in \mathbb{Q}_p$ satisfying 
	\begin{align*}
		\left|\sum_{j=1}^{d}\lambda_j\right|^2\leq |d| \left|\sum_{j=1}^{d}\lambda_j^2\right|.	
	\end{align*}
	\end{enumerate}
\end{definition}
To set up Definition \ref{EL}, we are mainly motivated by the following observations.
\begin{enumerate}
	\item Recall that  a basis  $\{\tau_j\}_{j=1}^n$  for  $\mathbb{Q}^d_p$ is said to be an orthonormal basis if $\langle \tau_j, \tau_k\rangle =\delta_{j,k}$ for all $1\leq j,k \leq n$. We then naturally have 
	\begin{align*}
		\sum_{j=1}^n\langle x, \tau_j\rangle \tau_j =x, \quad \forall x \in \mathbb{Q}_p^d.
	\end{align*}
\item In analogy to tight frames in Hilbert spaces, a generalization of orthonormal bases is the following:  A collection $\{\tau_j\}_{j=1}^n$  in   $\mathbb{Q}^d_p$ is said to be  a $p$-adic tight frame for $\mathbb{Q}^d_p$  if there is a nonzero element $b \in \mathbb{Q}_p$ such that 
\begin{align*}
	\sum_{j=1}^n\langle x, \tau_j\rangle \tau_j =bx, \quad \forall x \in \mathbb{Q}_p^d.
\end{align*}
Note that the operator corresponding to a tight frame operator is a scalar operator and is diagonalizable. Since equiangular lines correspond to tight frames in both real and complex Hilbert spaces, and diagonalization is satisfied by symmetric operators like those ones given by the formula in condition (iii), it is reasonable to impose diagonalization in the definition of $p$-adic equiangular lines.
\end{enumerate}
The most general version of $p$-adic equiangular lines can be obtained by considering Definition \ref{EL} without condition (iii). However, we are unable to study this level of generality.\\
The result of this  paper is the following $p$-adic version of Theorem \ref{VSTHEOREM}.
\begin{theorem} (\textbf{$p$-adic van Lint-Seidel  Relative Bound})\label{PE}
	Let $p$ be a prime, $d \in \mathbb{N}$ and $\gamma \geq 0$. If $\{\tau_j\}_{j=1}^n$  is a collection of $p$-adic  $\gamma$-equiangular lines in $\mathbb{Q}^d_p$, then 
	\begin{align*}
		|n|^2\leq |d|\max\{|n|, \gamma^2 \}.
	\end{align*}
In particular, we have the following. 
\begin{enumerate}[\upshape(i)]
	\item If $|n|\leq \gamma^2$, then 
	\begin{align*}
		|n|^2\leq |d|\gamma^2.
	\end{align*}
\item If $|n|\geq \gamma^2$, then 
\begin{align*}
	|n|\leq |d|.
\end{align*}
\end{enumerate}
\end{theorem}
\begin{proof}
   We see that 
\begin{align*}
	&\operatorname{Tra}(S_{\tau})=\sum_{j=1}^n\langle \tau_j, \tau_j \rangle , \\
	& \operatorname{Tra}(S^2_{\tau})=\sum_{j=1}^n\sum_{k=1}^n\langle \tau_j, \tau_k \rangle\langle \tau_k, \tau_j \rangle=\sum_{j=1}^n\sum_{k=1}^n\langle \tau_j, \tau_k \rangle^2.
\end{align*}	
Using Definition \ref{EL}, we get
\begin{align*}
	|n|^2&= \left|\sum_{j=1}^n\langle \tau_j, \tau_j \rangle \right|^2=|\operatorname{Tra}(S_{\tau})|^2=\left|\sum_{j=1}^{d}\lambda_j\right|^2\leq |d| \left|\sum_{j=1}^{d}\lambda_j^2\right|\\
	&=|d|\left|\sum_{j=1}^n\sum_{k=1}^n\langle \tau_j, \tau_k \rangle^2\right|
	=|d|\left|\sum_{j=1}^n\langle \tau_j, \tau_j \rangle^2+\sum_{1\leq j, k \leq n, j \neq k}\langle \tau_j, \tau_k \rangle^2\right|\\
	&=|d|\left|\sum_{j=1}^n1+\sum_{1\leq j, k \leq n, j \neq k}\langle \tau_j, \tau_k \rangle^2\right|=|d|\left|n+\sum_{1\leq j, k \leq n, j \neq k}\langle \tau_j, \tau_k \rangle^2\right|\\
	&\leq |d| \max\left\{|n|, \left|\sum_{1\leq j, k \leq n, j \neq k}\langle \tau_j, \tau_k \rangle^2\right|\right\}\\
	&\leq |d| \max\left\{|n|, \max_{1\leq j,k \leq n, j \neq k}|\langle \tau_j,\tau_k\rangle |^2\right\}= |d|\max\{|n|, \gamma^2 \}.
\end{align*}
\end{proof}
\begin{corollary}
Let $\{\tau_j\}_{j=1}^n$ be a collection  in $\mathbb{Q}^d_p$ satisfying the following.	
\begin{enumerate}[\upshape(i)]
	\item $\langle \tau_j, \tau_j\rangle =1, \quad \forall 1\leq j \leq n$.
	\item 	There exists a $\gamma\geq0$ such that $|\langle \tau_j, \tau_k\rangle| =\gamma, \forall 1\leq j, k \leq n, j \neq k$.
	\item There exists a nonzero element  $b \in \mathbb{Q}_p$ such that 
	\begin{align*}
		bx=\sum_{j=1}^{n}\langle x, \tau_j\rangle \tau_j, \quad \forall x \in \mathbb{Q}^d_p.
	\end{align*} 
Then 	
\begin{align*}
	|n|^2\leq |d|\max\{|n|, \gamma^2 \}.
\end{align*}
\end{enumerate}
\end{corollary}
A careful observation of the proof of Theorem \ref{PE} gives the following general $p$-adic Welch bound.
\begin{theorem} (\textbf{General $p$-adic Welch Bound})
Let $\{\tau_j\}_{j=1}^n$ be a collection  in $\mathbb{Q}^d_p$ satisfying the following. 
\begin{enumerate}[\upshape(i)]
	\item $\langle \tau_j, \tau_j\rangle =1, \quad \forall 1\leq j \leq n$.
	\item The operator 
	\begin{align*}
		S_\tau : \mathbb{Q}_p^d\ni x \mapsto \sum_{j=1}^n\langle x, \tau_j\rangle \tau_j \in \mathbb{Q}_p^d
	\end{align*}
	is similar  to a diagonal operator over $\mathbb{Q}_p$  with eigenvalues $\lambda_1, \dots, \lambda_d \in \mathbb{Q}_p$ satisfying 
	\begin{align*}
		\left|\sum_{j=1}^{d}\lambda_j\right|^2\leq |d| \left|\sum_{j=1}^{d}\lambda_j^2\right|.	
	\end{align*}
\end{enumerate}
Then 
\begin{align*}
	|n|^2\leq |d|\max_{1\leq j,k \leq n, j \neq k}\left\{|n|, |\langle \tau_j,\tau_k\rangle |^2\right\}.
\end{align*}
\end{theorem}
We can generalize   Definition \ref{EL} in the following way. 
\begin{definition}\label{A}
	Let $p$ be a prime, $d \in \mathbb{N}$, $\gamma \geq0$ and $a\in \mathbb{Q}_p$ be nonzero. A collection $\{\tau_j\}_{j=1}^n$  in $\mathbb{Q}_p^d$	is said to be \textbf{$p$-adic $(\gamma,a)$-equiangular  lines} if the following conditions hold.
\begin{enumerate}[\upshape(i)]
	\item $\langle \tau_j, \tau_j\rangle =a, \quad \forall 1\leq j \leq n$.
	\item $	|\langle \tau_j, \tau_k\rangle| =\gamma, \forall 1\leq j, k \leq n, j \neq k$.
	\item The operator 
	\begin{align*}
		S_\tau : \mathbb{Q}_p^d\ni x \mapsto \sum_{j=1}^n\langle x, \tau_j\rangle \tau_j \in \mathbb{Q}_p^d
	\end{align*}
	is similar  to a diagonal operator over $\mathbb{Q}_p$  with eigenvalues $\lambda_1, \dots, \lambda_d \in \mathbb{Q}_p$ satisfying 
	\begin{align*}
		\left|\sum_{j=1}^{d}\lambda_j\right|^2\leq |d| \left|\sum_{j=1}^{d}\lambda_j^2\right|.	
	\end{align*}
\end{enumerate}	
\end{definition}
Note that division by norm of an element is not allowed in a $p$-adic Hilbert space. 
Thus we can not  reduce Definition \ref{A} to Definition \ref{EL} (unlike in the real case). 
By modifying the proof of Theorem \ref{PE}, we easily get the following results. 
\begin{theorem}
	 If $\{\tau_j\}_{j=1}^n$  is a collection of $p$-adic  $(\gamma,a)$-equiangular lines in $\mathbb{Q}^d_p$, then 
	\begin{align*}
		|n|^2\leq |d|\max\left\{|n|, \frac{\gamma^2}{|a^2|} \right\}.
	\end{align*}
	In particular, we have the following. 
	\begin{enumerate}[\upshape(i)]
		\item If $|a^2n|\leq \gamma^2$, then 
		\begin{align*}
			|n|^2\leq |d|\frac{\gamma^2}{|a^2|}.
		\end{align*}
		\item If $|a^2n|\geq \gamma^2$, then 
		\begin{align*}
			|n|\leq |d|.
		\end{align*}
	\end{enumerate}
\end{theorem}
\begin{corollary}
	Let $\{\tau_j\}_{j=1}^n$ be a collection  in $\mathbb{Q}^d_p$ satisfying the following.	
	\begin{enumerate}[\upshape(i)]
		\item There exists a nonzero element  $a \in \mathbb{Q}_p$ such that  $\langle \tau_j, \tau_j\rangle =a,  \forall 1\leq j \leq n$.
		\item There exists a $\gamma\geq 0$ such that $|\langle \tau_j, \tau_k\rangle| =\gamma, \forall 1\leq j, k \leq n, j \neq k$.
		\item There exists a nonzero element  $b \in \mathbb{Q}_p$ such that 
		\begin{align*}
			bx=\sum_{j=1}^{n}\langle x, \tau_j\rangle \tau_j, \quad \forall x \in \mathbb{Q}^d_p.
		\end{align*} 
		Then 	
	\begin{align*}
		|n|^2\leq |d|\max\left\{|n|, \frac{\gamma^2}{|a^2|} \right\}.
	\end{align*}
	\end{enumerate}
\end{corollary}
\begin{theorem}
Let $\{\tau_j\}_{j=1}^n$ be a collection  in $\mathbb{Q}^d_p$ satisfying the following. 
\begin{enumerate}[\upshape(i)]
	\item There exists a nonzero element  $a \in \mathbb{Q}_p$ such that  $\langle \tau_j, \tau_j\rangle =a,  \forall 1\leq j \leq n$.
	\item The operator 
	\begin{align*}
		S_\tau : \mathbb{Q}_p^d\ni x \mapsto \sum_{j=1}^n\langle x, \tau_j\rangle \tau_j \in \mathbb{Q}_p^d
	\end{align*}
	is similar  to a diagonal operator over $\mathbb{Q}_p$  with eigenvalues $\lambda_1, \dots, \lambda_d \in \mathbb{Q}_p$ satisfying 
	\begin{align*}
		\left|\sum_{j=1}^{d}\lambda_j\right|^2\leq |d| \left|\sum_{j=1}^{d}\lambda_j^2\right|.	
	\end{align*}
\end{enumerate}
Then 
\begin{align*}
	|n|^2\leq |d|\max_{1\leq j,k \leq n, j \neq k}\left\{|n|, \frac{|\langle \tau_j,\tau_k\rangle |^2}{|a^2|}\right\}.
\end{align*}	
\end{theorem}
Note that there is a universal bound for equiangular lines known as Gerzon bound. 
\begin{theorem} \cite{WALDRONBOOK} \label{GBT}
(\textbf{Gerzon Universal Bound})
Let $\{\tau_j\}_{j=1}^n$  	be  a collection of $\gamma$-equiangular  lines in $\mathbb{R}^d$. 
	Then 
	\begin{align*}
	n \leq \frac{d(d+1)}{2}.
\end{align*}	
\end{theorem}
We are unable to derive $p$-adic version of Theorem \ref{GBT}.  However, we derive the following  Gerzon bound under some additional mild conditions.
\begin{theorem}	\label{PGERZON}
	Let $\{\tau_j\}_{j=1}^n$ be a collection in $\mathbb{Q}^d_p$  satisfying the following.
	\begin{enumerate}[\upshape(i)]
		\item There exists a nonzero element  $a \in \mathbb{Q}_p$ such that  $\langle \tau_j, \tau_j\rangle =a,  \forall 1\leq j \leq n$.
		\item There is an element  $b \in \mathbb{Q}_p$ such that $\langle \tau_j, \tau_k\rangle^2 =b, \forall  1\leq j,k \leq n, j\neq k.$
		\item $a^2\neq b$. 
	\end{enumerate}
	Then 
	\begin{align*}
		n \leq  \frac{d(d+1)}{2}.	
	\end{align*}
\end{theorem}
\begin{proof}
	For $1\leq j \leq n$, define 
	\begin{align*}
		\tau_j \otimes \tau_j: \mathbb{Q}^d_p \ni x \mapsto(\tau_j \otimes \tau_j)(x)\coloneqq \langle x, \tau_j\rangle \tau_j \in \mathbb{Q}^d_p.
	\end{align*}
	We wish to show that the collection $\{\tau_j \otimes \tau_j\}_{j=1}^n$   is linearly independent  over  $ \mathbb{Q}_p$. Let $c_1, \dots,  c_n \in  \mathbb{Q}_p$ be such that 
	\begin{align*}
		\sum_{j=1}^{n}c_j(\tau_j \otimes \tau_j)=0.
	\end{align*}
	Let $1\leq k \leq n$ be fixed. Then the previous equation gives 
	\begin{align*}
		\sum_{j=1}^{n}c_j(\tau_j \otimes \tau_j)(\tau_k \otimes \tau_k)	=0.
	\end{align*}
	By taking trace we get
	\begin{align*}
		0&=\sum_{j=1}^{n}c_j\operatorname{Tra}((\tau_j \otimes \tau_j)(\tau_k \otimes \tau_k))=\sum_{j=1}^{n}c_j\langle \tau_j, \tau_k\rangle^2\\
		&=\sum_{j=1, j \neq k}^{n}c_j\langle \tau_j, \tau_k\rangle^2+c_k\langle \tau_k, \tau_k\rangle^2=\sum_{j=1, j \neq k}^{n}c_jb+c_ka^2\\
		&=b\left(\sum_{j=1}^{n}c_j-c_k\right)+c_ka^2=b\left(\sum_{j=1}^{n}c_j\right)+(a^2-b)c_k.
	\end{align*}
	Therefore, 
	\begin{align*}
		c_k=\frac{b}{b-a^2}\sum_{j=1}^{n}c_j=:c, \quad \forall 1\le k \leq n.	
	\end{align*}
	Finally, we get 
	\begin{align*}
		0&=\operatorname{Tra}\left(\sum_{j=1}^{n}c_j(\tau_j \otimes \tau_j)\right)=\operatorname{Tra}\left(\sum_{j=1}^{n}c(\tau_j \otimes \tau_j)\right)\\
		&=\sum_{j=1}^{n}c\operatorname{Tra}(\tau_j \otimes \tau_j)=\sum_{j=1}^{n}c\langle \tau_j, \tau_j\rangle=can.
	\end{align*}
	Since $a\neq 0$, we have $c=0$. Therefore, $\{\tau_j \otimes \tau_j\}_{j=1}^n$ is linearly independent.  Since $\{\tau_j \otimes \tau_j\}_{j=1}^n$ is given by a symmetric $n$ by $n$ matrix over $\mathbb{Q}_p$ (w.r.t. the standard basis) and the dimension of vector space of symmetric $n$ by $n$ matrices  over $\mathbb{Q}_p$ is $d(d+1)/2$,  we must have $n\leq d(d+1)/2$.
\end{proof}
It is clear that throught the paper we can replace $\mathbb{Q}_p$ by any non-Archimedean field.

\section{Acknowledgments}
This research was partially supported by the University of Warsaw Thematic Research Programme ``Quantum Symmetries". Theorem \ref{PGERZON} is motivated by a question raised by the anonymous reviewer. We thank the reviewer for his/her many questions and suggestions, which helped to substantially improve the paper.

 \bibliographystyle{plain}

\end{document}